\documentclass[11pt,reqno]{amsart}
\usepackage[latin1]{inputenc} 
\usepackage{indentfirst}
\usepackage{amsmath} 
\usepackage{amsfonts} 
\usepackage{amsthm} 
\usepackage{epsfig} 
\usepackage{graphicx} 
\usepackage{multirow} 
\usepackage{verbatim} 
\usepackage{rotating} 
\usepackage{lettrine}
\usepackage{amscd}
\usepackage{pstricks}
\usepackage{pst-tree}
\usepackage{pst-node}
\usepackage{dsfont}
\usepackage{lscape}
\usepackage{ae}
\usepackage{float}
\usepackage[all]{xy}
\usepackage{icomma} 
\usepackage{amsmath} 
\usepackage{amsfonts} 
\usepackage{amsthm} 
\usepackage{varioref} 
\usepackage{layout}
\usepackage[Lenny]{fncychap} 
\usepackage{enumerate}
\usepackage{amsfonts, amsmath, amssymb, stmaryrd, latexsym}
\usepackage{dsfont}
\usepackage{array}
\usepackage{longtable}
\usepackage{geometry}
\usepackage[frenchb,english]{babel}

\theoremstyle{plain}
\newtheorem{them}{Théorème}[]
\newtheorem{propo}{Proposition}[]

\newtheorem{lem}{Lemme}[]
\theoremstyle{remark}

\def\NN{\mathds{N}}

\def\QQ{\mathbb{Q}}

\def\ZZ{\mathbb{Z}}

\def\l{\mathds{L}}
\addto\captionsfrench{}
\begin{document}

\title[ Sur l'équation...]{Sur l'équation $X^2-\varepsilon_2\varepsilon_{p_1p_2}\varepsilon_{2p_1p_2}=0$.}

\author[Abdelmalek AZIZI]{Abdelmalek Azizi}
\address{Abdelmalek Azizi et Abdelkader Zekhnini: Département de Mathématiques, Faculté des Sciences, Université Mohammed 1, Oujda, Morocco }
\author{Abdelkader Zekhnini}
\email{abdelmalekazizi@yahoo.fr}
\email{zekha1@yahoo.fr}

\author[Mohammed Taous]{Mohammed Taous}
\address{Mohammed Taous: Département de Mathématiques, Faculté des Sciences et Techniques, Université Moulay Ismail, Errachidia, Morocco}
\email{taousm@hotmail.com}
\maketitle

\begin{abstract}
 Let $p_1\equiv p_2\equiv5 \pmod8$ be prime numbers such that $\left(\frac{p_1}{p_2}\right)=-1$. Let $\l=\QQ(\sqrt2, \sqrt{p_1p_2})$ Our goal is to  resolve the equation $X^2-\varepsilon_2\varepsilon_{p_1p_2}\varepsilon_{2p_1p_2}=0$ in $\l$, where  $\varepsilon_j$ are fundamental units of real quadratic subfields  of $\QQ(\sqrt2, \sqrt{p_1p_2})$.
\end{abstract}
\selectlanguage{francais}
\begin{abstract}
 Soient $p_1$ et $p_2$ deux nombres premiers tels que $p_1\equiv p_2\equiv5 \pmod
 8$ et $\left(\frac{p_1}{p_2}\right)=-1$; posons $\l=\QQ(\sqrt2, \sqrt{p_1p_2})$. Dans ce papier notre but est de caractériser, par les symboles biquadratiques,
 l'existence de solutions de l'équation $X^2-\varepsilon_2\varepsilon_{p_1p_2}\varepsilon_{2p_1p_2}=0$ dans $\l$, où les $\varepsilon_j$ sont les unités fondamentales des sous-corps quadratiques réels de $\l$.
 \end{abstract}
 \section{Introduction}
 Soient  $d_1$, $d_2$ deux entiers naturels différents et sans facteurs carrés,  $K=\QQ(\sqrt{d_1}, \sqrt{d_2})$, $\varepsilon_1$ (resp. $\varepsilon_2$, $\varepsilon_3$) l'unité fondamentale de $k_1=\QQ(\sqrt{d_1})$ (resp. $k_2=\QQ(\sqrt{d_2})$, $k_3=\QQ(\sqrt{d_1d_2})$). On suppose que $\varepsilon_1$, $\varepsilon_2$ et $\varepsilon_3$ sont de norme -1. Nombreux sont les mathématiciens qui ont travaillé sur les unités du corps $K$ et ont donné le système fondamental d'unités (SFU) de $K$ en fonction des trois unités $\varepsilon_1$, $\varepsilon_2$ et $\varepsilon_3$, mais le problème était de connaître quand est ce que le produit $\varepsilon_1\varepsilon_2\varepsilon_3$ est un carré dans $K$. En posant
 \begin{align*}
 \alpha_1&=\varepsilon_1\varepsilon_2\varepsilon_3+\varepsilon_1+\varepsilon_2-\varepsilon_3\\ \alpha_2&=\varepsilon_1\varepsilon_2\varepsilon_3+\varepsilon_1-\varepsilon_2+\varepsilon_3\\ \alpha_3&=\varepsilon_1\varepsilon_2\varepsilon_3-\varepsilon_1+\varepsilon_2+\varepsilon_3\\  \alpha_4&=\varepsilon_1\varepsilon_2\varepsilon_3-\varepsilon_1-\varepsilon_2-\varepsilon_3\\ c_j&=trace_{K/\QQ}(\alpha_j), \text{ où } j\in J=\{1, 2, 3, 4\}.
 \end{align*}
 T. Kubota dans \cite{Kub-56} a prouvé que $c_j=\frac{\alpha_j^2}{\varepsilon_1\varepsilon_2\varepsilon_3}$ pour tout $j\in J$, donc $\varepsilon_1\varepsilon_2\varepsilon_3$ est un carré dans $K$ si et seulement s'il existe un $j\in J$ tel que $c_j$ est un carré dans $K$. Sur ce résultat est basé A. Azizi dans \cite{Az-05} pour conclure que  $\varepsilon_1\varepsilon_2\varepsilon_3$ est un carré dans $K$ ssi il existe $(i, j)\in \{0, 1\}^2$ tel que $d_1^{i}d_2^{j}trace_{K/\QQ}(\alpha_1)$ est un carré dans $\NN$; on trouve la même chose chez M. Hirabayashi et K. Yoshino dans \cite{HiYo} qui ont conclu que   $\varepsilon_1\varepsilon_2\varepsilon_3$ est un carré dans $K$ ssi  $trace_{K/\QQ}(\alpha_3)$ est un carré dans $K$; de même H. Wada a utilisé dans \cite{Wa-66} les $trace_{K/\QQ}(\alpha_j)$ pour trouver dans quel cas $\varepsilon_1\varepsilon_2\varepsilon_3$ est un carré dans $K$.   Dans ce papier on répond à cette question on se  basant sur les symboles rationales biquadratiques dans le cas $d_1=2$ et $d_2=p_1p_2$, où $p_1$ et $p_2$ sont des nombres premiers.\par
 Soient $p_1$ et $p_2$ deux numbers premiers tels que $p_1\equiv p_2\equiv5 \pmod8$, $\left(\frac{p_1}{p_2}\right)=-1$. Posons  $\l=\QQ(\sqrt2, \sqrt{p_1p_2})$, $\varepsilon_2$ ( resp. $\varepsilon_{p_1p_2}$, $\varepsilon_{2p_1p_2}$ ) l'unité fondamentale de $k_1=\QQ(\sqrt2)$ (resp.  $k_2=\QQ(\sqrt{p_1p_2})$, $k_3=\QQ(\sqrt{2p_1p_2})$), alors  $N(\varepsilon_2)=N(\varepsilon_{p_1p_2})=N(\varepsilon_{2p_1p_2})=-1$ (voir Lemme \ref{2}), donc d'après S. Kuroda \cite{Kur-43}  un SFU de $\l=\QQ(\sqrt2, \sqrt{p_1p_2})$ est  $\{\varepsilon_2, \varepsilon_{p_1p_2}, \sqrt{\varepsilon_2\varepsilon_{p_1p_2}\varepsilon_{2p_1p_2}}\}$  ou bien $\{\varepsilon_2, \varepsilon_{p_1p_2}, \varepsilon_{2p_1p_2}\}$ suivant que l'équation $X^2-\varepsilon_2\varepsilon_{p_1p_2}\varepsilon_{2p_1p_2}=0$ admet ou non une solution dans $\l$, dans le théorème \ref{4}, on a utilisé les symboles biquadratiques pour déterminer dans quel cas l'unité $\varepsilon_2\varepsilon_{p_1p_2}\varepsilon_{2p_1p_2}$ est un carré dans $\l$ et nous avons  donné une relation entre cette réponse et le $2$-nombre de classes de $\QQ(\sqrt{-p_1p_2})$.
  \section{Le résultat.}
   On commence  par donner quelques résultats qui nous serons utiles par la suite. Soient $a$ un entier naturel sans facteurs carrés et $Q$ l'indice d'unités du corps $\QQ(\sqrt a,i)$, alors on a :
\begin{lem}\label{1}
Si l'une des conditions suivantes est vérifiée, alors $Q=1.$
\begin{enumerate}
  \item $a$ est congru à 1 modulo 4.
  \item Il existe un entier impair $a'$ qui divise $a$ tel que $a'\equiv5\pmod 8.$
\end{enumerate}
\end{lem}
\proof Voir corollaire 3.2 de \cite{AzTa08}.
\begin{lem}\label{2}
Si $d=2p_1p_2$ avec $p_1$ et $p_2$ sont deux nombres  premiers tels
que $p_1\equiv p_2\equiv1\pmod 4$ et  au moins deux des éléments \{$(\frac{p_1}{p_2})$, $(\frac{2}{p_1})$, $(\frac{2}{p_2})$\} valent -1, alors la norme de l'unité fondamentale de $\QQ(\sqrt{2p_1p_2})$ est égale à -1.
\end{lem}
\proof Voir corollaire 3.6. de \cite{AzTa08}.
\begin{propo}\label{3}
Soient $p_1$ et $p_2$ deux nombres premiers tels que $p_1\equiv p_2\equiv1 \pmod4$ et $\displaystyle\left(\frac{2}{p_1}\right)=\displaystyle\left(\frac{2}{p_2}\right)=\displaystyle\left(\frac{p_1}{p_2}\right)$, posons $p_1=\pi_1\pi_2=(a_1+2ib_1)(a_1-2ib_1)$ et $p_2=\pi_3\pi_4=(a_2+2ib_2)(a_2-2ib_2)$. Alors on a:
$$\displaystyle\left(\frac{p_1p_2}{2}\right)_4\displaystyle\left(\frac{2p_1}{p_2}\right)_4
\displaystyle\left(\frac{2p_2}{p_1}\right)_4=\displaystyle\left(\frac{\pi_3}{\pi_1}\right)
\displaystyle\left(\frac{2}{a_1+2b_1}\right)\displaystyle\left(\frac{2}{a_2+2b_2}\right).$$
\end{propo}
\begin{proof}
Posons $p_1=\pi_1\pi_2=(a_1+2ib_1)(a_1-2ib_1)$ et $p_2=\pi_3\pi_4=(a_2+2ib_2)(a_2-2ib_2)$, alors par la loi de réciprocité biquadratique  on a:
\begin{align*}
\displaystyle\left(\frac{2p_1}{p_2}\right)_4\displaystyle\left(\frac{2p_2}{p_1}\right)_4
&=\displaystyle\left(\frac{2p_1}{\pi_3}\right)_4\displaystyle\left(\frac{2p_2}{\pi_1}\right)_4\\
&=\displaystyle\left(\frac{2}{\pi_1\pi_3}\right)_4\displaystyle\left(\frac{p_1}{\pi_3}\right)_4\displaystyle\left(\frac{p_2}{\pi_1}\right)_4\\
&=\displaystyle\left(\frac{2}{\pi_1\pi_3}\right)_4\displaystyle\left(\frac{\pi_3}{\pi_1}\right).
\end{align*}
Puisque $2=i^3(1+i)^2$, alors
\begin{align*}
 \displaystyle\left(\frac{2}{\pi_1}\right)_4\displaystyle\left(\frac{2}{\pi_3}\right)_4
 &=\displaystyle\left(\frac{i^3(1+i)^2}{\pi_1}\right)_4\displaystyle\left(\frac{i^3(1+i)^2}{\pi_3}\right)_4\\
 &=\displaystyle\left(\frac{i}{\pi_1}\right)_4\displaystyle\left(\frac{1+i}{\pi_1}\right)
 \displaystyle\left(\frac{i}{\pi_3}\right)_4\displaystyle\left(\frac{1+i}{\pi_3}\right)\\
 &=i^{\frac{p_1+p_2-2}{4}}\displaystyle\left(\frac{1+i}{\pi_1}\right)\displaystyle\left(\frac{1+i}{\pi_3}\right)\\
 &=i^{\frac{p_1+p_2-2}{4}}\displaystyle\left(\frac{2}{a_1+2b_1}\right)\displaystyle\left(\frac{2}{a_2+2b_2}\right)
 \end{align*}
 D'autre part comme  $\frac{p_1-1}{8}+\frac{p_2-1}{8}\equiv\frac{p_1p_2-1}{8}\pmod2$, alors on a:
$$\displaystyle\left(\frac{p_1p_2}{2}\right)_4=(-1)^{\frac{p_1p_2-1}{8}}=(-1)^{{\frac{p_1+p_2-2}{8}}}=i^{\frac{p_1+p_2-2}{4}}$$ d'où $$ \displaystyle\left(\frac{2}{\pi_1}\right)_4\displaystyle\left(\frac{2}{\pi_3}\right)_4=
\displaystyle\left(\frac{p_1p_2}{2}\right)_4\displaystyle\left(\frac{2}{a_1+2b_1}\right)\displaystyle\left(\frac{2}{a_2+2b_2}\right).$$
Et le résultat en découle.
\end{proof}
\begin{them}\label{4}
Soient $p_1$ et $p_2$ deux nombres premiers tels que $p_1\equiv p_2\equiv5\pmod8$ et  $\displaystyle\left(\frac{p_1}{p_2}\right)=-1$, posons $\l=\QQ(\sqrt2, \sqrt{p_1p_2})$. Alors les assertions suivantes sont équivalentes:
\begin{enumerate}[\upshape\indent(1)]
  \item $\varepsilon_2\varepsilon_{p_1p_2}\varepsilon_{2p_1p_2}$ est un carré dans $\l$.
  \item $\displaystyle\left(\frac{p_1p_2}{2}\right)_4\displaystyle\left(\frac{2p_1}{p_2}\right)_4
\displaystyle\left(\frac{2p_2}{p_1}\right)_4=-1.$
  \item L'indice d'unités du corps $\l$ est égal à $q(\l/\QQ)=2$.
  \item $\mathbf{C}_2(\QQ(\sqrt{-p_1p_2}))$, le $2$-groupe de classes de $\QQ(\sqrt{-p_1p_2})$, est de type $(2, 4)$.
  \item Le  $2$-nombre de classes de $\QQ(\sqrt{-p_1p_2})$ est égal à $h_2(-p_1p_2)=8$.
  \item Le nombre de classes de $\QQ(\sqrt{2p_1p_2})$  capitulant dans  $\QQ(\sqrt2, \sqrt{p_1p_2})$ est $2$.
\end{enumerate}
\end{them}
\begin{proof}
Comme $p_1p_2\equiv1 \pmod8$, alors  il existe un entier pair $x$ et un entier impair  $y$  tels que $\varepsilon_{p_1p_2}=x+y\sqrt{p_1p_2}$, donc $x^2+1=y^2p_1p_2$, par suite d'après \cite{Kw-80}:
 \begin{equation}
 \label{8}
\left\{
 \begin{array}{ll}
 x\mp i=iy_1^2\pi_1\pi_3\\
 x\pm i=-iy_2^2\pi_2\pi_4
 \end{array}\right. \text{ (\ref{8}1) ou }
 \left\{
 \begin{array}{ll}
 x\mp i=iy_1^2\pi_1\pi_4\\
 x\pm i=-iy_2^2\pi_2\pi_3
 \end{array}\right. \text{ (\ref{8}2) }
 \end{equation}
 d'où \\
 \begin{equation}
 \label{9}\left.
 \begin{array}{ll}\sqrt{\varepsilon_{p_1p_2}}=z_1\sqrt{\pi_1\pi_3}+z_2\sqrt{\pi_2\pi_4}\ \text{ (\ref{9}1) }\\
 \text{\   ou } \\ \sqrt{\varepsilon_{p_1p_2}}=z_1\sqrt{\pi_1\pi_4}+z_2\sqrt{\pi_2\pi_3}\  \text{\ (\ref{9}2) } \end{array} \right\}\end{equation}
 Avec $z_2$ est le conjugué de $z_1$ dans $\frac{1}{2}\ZZ[i]$.\\
 \indent De même il existe $a$ et $b$ dans $\NN$ tels que  $\varepsilon_{2p_1p_2}=a+b\sqrt{2p_1p_2}$, donc
 \begin{equation}\label{10}
   \left.
 \begin{array}{ll}
  \left\{
 \begin{array}{ll}
 a\mp i=(1+i)b_1^2\pi_1\pi_3\\
 x\pm i=(1-i)b_2^2\pi_2\pi_4
 \end{array}\right. \text{(\ref{10}1) ou}
 \left\{
 \begin{array}{ll}
 x\mp i=(1+i)b_1^2\pi_1\pi_4\\
 x\pm i=(1-i)b_2^2\pi_2\pi_3
 \end{array}\right. \text{(\ref{10}2) ou} \\
 \left\{
 \begin{array}{ll}
 a\mp i=i(1+i)b_1^2\pi_1\pi_3\\
 x\pm i=-i(1-i)b_2^2\pi_2\pi_4
 \end{array}\right. \text{(\ref{10}3) ou}
 \left\{
 \begin{array}{ll}
 x\mp i=i(1+i)b_1^2\pi_1\pi_4\\
 x\pm i=-i(1-i)b_2^2\pi_2\pi_3
 \end{array}\right.  \text{(\ref{10}4)}\end{array}\right\}\end{equation}
  par suite
  \begin{equation}
 \label{11}\left.
 \begin{array}{ll}\sqrt{2\varepsilon_{2p_1p_2}}=u_1\sqrt{(1+i)\pi_1\pi_3}+u_2\sqrt{(1-i)\pi_2\pi_4}\
 \text{ (\ref{11}1) ou } \\ \sqrt{2\varepsilon_{2p_1p_2}}=u_1\sqrt{(1+i)\pi_1\pi_4}+u_2\sqrt{(1-i)\pi_2\pi_3}\  \text{\ (\ref{11}2) ou} \\
 \sqrt{\varepsilon_{2p_1p_2}}=u_1\sqrt{(1+i)\pi_1\pi_3}+u_2\sqrt{(1-i)\pi_2\pi_4}\
 \text{ (\ref{11}3) ou } \\ \sqrt{\varepsilon_{2p_1p_2}}=u_1\sqrt{(1+i)\pi_1\pi_4}+u_2\sqrt{(1-i)\pi_2\pi_3}\  \text{\ (\ref{11}4) }
 \end{array} \right\}\end{equation}
 où $u_2$ et $u_1$ sont conjugués dans $\ZZ[i]$ ou $\frac{1}{2}\ZZ[i]$.
 Enfin on a:
  \begin{equation}\label{12}
  \sqrt{2\varepsilon_2}=\sqrt{1+i}+\sqrt{1-i}.
  \end{equation}
 En multipliant les égalités (\ref{9}1), (\ref{11}1) et (\ref{12}); (\ref{9}1), (\ref{11}3) et (\ref{12}); (\ref{9}2), (\ref{11}2) et (\ref{12}); (\ref{9}2), (\ref{11}4) et (\ref{12}) on trouve que $\varepsilon_2\varepsilon_{p_1p_2}\varepsilon_{2p_1p_2}$ est un carré dans $\l$. Donc $\varepsilon_2\varepsilon_{p_1p_2}\varepsilon_{2p_1p_2}$ est un carré dans $\l$ si et seulement si $x$ et $a$ prennent l'une des formes suivantes:\\
  (\ref{8}1) et  (\ref{10}1) ou  (\ref{8}1) et  (\ref{10}3) ou  (\ref{8}2) et  (\ref{10}2) ou  (\ref{8}2) et  (\ref{10}4).\\
 On applique le symbole des restes quadratiques à ces quatre formes on trouve que: $$\displaystyle\left(\frac{\pi_3}{\pi_1}\right)\displaystyle\left(\frac{2}{a_1+2b_1}
 \right)\displaystyle\left(\frac{2}{a_2+2b_2}\right)=
\displaystyle\left(\frac{\pi_3}{\pi_1}\right)\displaystyle\left(\frac{1+i}{\pi_1}\right)
 \displaystyle\left(\frac{1+i}{\pi_3}\right)=-1.$$
 Et la proposition \ref{3} nous donne $$\displaystyle\left(\frac{p_1p_2}{2}\right)_4\displaystyle\left(\frac{2p_1}{p_2}\right)_4
\displaystyle\left(\frac{2p_2}{p_1}\right)_4=-1$$
 Par contre les autres formes nous donnent que:
  $$\displaystyle\left(\frac{p_1p_2}{2}\right)_4\displaystyle\left(\frac{2p_1}{p_2}\right)_4
\displaystyle\left(\frac{2p_2}{p_1}\right)_4=
\displaystyle\left(\frac{\pi_3}{\pi_1}\right)\displaystyle\left(\frac{1+i}{\pi_1}\right)
 \displaystyle\left(\frac{1+i}{\pi_3}\right)=1.$$
 Ceci établit l'équivalence entre les assertions (1) et (2).\\
 \indent L'équivalence entre (1) et (3) est une résultat bien connue.\\
 \indent Pour les équivalences $(2)\Leftrightarrow (4)$ et $(2)\Leftrightarrow (5)$ voir \cite{Ka76}.\\
 \indent Pour l'équivalence entre  (1) et (6) voir \cite{B.S.C-94}.
\end{proof}
 \bibliographystyle{alpha}

\begin{thebibliography}{xxxxxxx}
\bibitem[Az-99-1]{Az-99-1}
 A. Azizi, {\em Sur le $2$-groupe de classe d'idéaux de $\QQ(\sqrt{d},i)$, }{Rend. Circ. Mat. Palmero (2) {\bf 48}
 (1999), 71-92.}
 \bibitem[Az-99-2]{Az-99-2}
A. Azizi, {\em Unités de certains corps de nombres imaginaires et abéliens sur
$\QQ$}{, Ann. Sci. Math. Québec {\bf 23} (1999), no 1, 15-21.}
\bibitem[Az-00]{Az-00} A.Azizi, {\em Sur la capitulation des $2$-classes d'idéaux de
$k =\QQ(\sqrt{2pq},i)$, } {Acta Arith. {\bf 94} (2000), 383-399.}
\bibitem[Az-03]{Az-03} A.Azizi, {\em Construction de la tour des $2$-corps de classes de
    Hilbert de certains corps biquadratiques, }{Pac. J.Math. {\bf 208} (2003), 1-10.}
\bibitem[Az-04]{Az-04}
 A. Azizi,{\em  comptes rendus de la conférence internationnale Maroc-Québec $2003$ },
 {Théorie des nombres et applications,(2004) 32-38. }
 \bibitem[Az-05]{Az-05}
 A. Azizi, {\em Sur les unités de certains corps de nombres de degré sur $\QQ$, }
 {Ann. Sci. Math. Québec {\bf 29} (2005), no 2, 111-129.}
\bibitem[Az-99]{Az-99}
 A. Azizi, {\em Capitulation of the $2$-ideal Classes of $\QQ(\sqrt{d}, i)$
 Where $p_1$ and $p_2$ are primes such that $p_1\equiv 1$ mod 8, $p_2\equiv 5$ mod 8 and $(\frac{p_1}{p_2})=-1$,}
 { Lecture notes in pure and applied mathematics. vol. {\bf 208} (1999), 13-19. }

\bibitem[A-T-08]{AzTa08}
A. Azizi et M. Taous, {\em Determination des corps $\mathbf{k}=\QQ(\sqrt{d}, i)$ dont le $2$-groupes de classes est de type $(2, 4)$ ou $(2, 2, 2)$}{, Rend. Istit. Mat. Univ. Trieste. {\textbf{40}} (2008), 93-116.}

\bibitem[B-03]{B-03}
C. Batut, K. Belabas, D. Bernadi, H. Cohen, M. Olivier, {GP/PARI calculator}{ Version 2.2.6 (2003).}

\bibitem[B-C-69]{B-C-69} P. Barruccand, H. Cohn, {\em Note on prime of type
    $x^2+32y^2$, class number, and residuacity}{, J. reine angew. Math. {\bf 238}
    (1969), 67-70.}

\bibitem[B.S.C-94]{B.S.C-94} E. Benjamin, F. Sanborn and C. Snyder, {\em Capitulation in unramified quadratic
extensions of real quadratic number fields,}{ Glasgow J. Math. {\bf 36} (1994), 385-392.}
\bibitem[B.L.S-98]{B.L.S-98} E. Benjamin; F. Lemmermeyer; C. Snyder,   { Real Quadratic Fields with Abelian 2-Class Field twoer,} {Journal of Number Theory, Volume 73, Number 2, December 1998, pp. 182-194 (13).}

\bibitem[G.S-65]{G.Sh-65} E. Golod and I. Shafarevich, {On the class field towers, } { Izv. Akad. Nauk SSSR,
{\bf 28} (1964), 261-272 (in Russian);} {English translation: Amer. Math. Soc. Transl, 48
(1965), 91-102, Amer. Math. Soc., Providence, R.I.}

\bibitem[Gr-73]{Gr-73}
 G.Gras,{\em Sur les l-classes d'idéaux dans les extensions cycliques relatives de degré premier, }{l, Ann.Inst.Fourier {\bf 23}, fasc. 3 (1973), 1-48.}

 \bibitem[G.J-96]{G.Janu-96}
 Gerald. J. Janusz, {Algebraic Number Fields, } { Graduate Studies in Mathematics USA, Second EDITION  (1996). }

\bibitem[Hi]{Hi}
 D. Hilbert, {\em \"Uber die Theorie des relativquadratischen Zahlk\"orper,} { Math. Ann.
 {\bf 51}, (1899), 1-127.}

 \bibitem[Hi-96]{Hi96}
 M. HlRABAYASHI, {\em Unit indices of some imaginary composite quadratic fields II,} { P. J. OF MATHEMATICS
 Vol. {\bf 173}, No. 1,  (1996).}

 \bibitem[H-Y-96]{HiYo}
 M. Hirabayashi and K. Yoshino {\em Unit Indices of Imaginary Abelian Number Fielfs of Type $(2, 2, 2)$, }{J. N. Theory, {\bf 34}, No. 3 (1990), 346-361.}

\bibitem[H-S-82]{H-S-82} F. P. Heider, B. Schmithals, {Zur Kapitulation der
    Idealklassen in unverzweigten primzgklischen Erweiterungen, }{J. Reine Angew. Math. {\bf 366}(1982), 1-25.}

\bibitem[H-Q-11]{HQ-11}
 H. Jung and Q. Yue,{\em 8-Ranks of class groups of imaginary quadratic number fields and their densities, }{J. Korean Math. Soc. {\bf 48} (2011), No. 6, pp. 1249-1268.}


\bibitem[I.M.I-76]{I.M.I-76} I. Matin. Isaacs, {\em Character Theory of Finite Groups,}{ New York: Academic Press, (1976).}

\bibitem[Ka-76]{Ka76}
P. Kaplan, {\em Sur le $2$-groupe de classes d'idéaux des corps
quadratiques, }{J. Reine angew. Math. {\bf 283/284} (1976),
313-363.}

\bibitem[Lm-95]{Lm95}
F. Lemmermeyer, {\em Ideal class groups of cyclotomic number fields I}{,ACTA ARITHMETICA LXXII.4 (1995).}

\bibitem[Lm-00]{Lm00}
F. Lemmermeyer, {\em Reciprocity Laws}{, Springer Monographs in Mathematics, Springer-Verlag. Berlin 2000.}

\bibitem[M-P-R-95]{McPaRa-95} T. M. McCall, C. J. Parry, R. R. Ranalli, {\em On
    imaginary bicyclic biquadratic fields with cyclic $2$-class group}{, J. Number
    Theory {\bf 53}, 88-99 (1995).}
\bibitem[Mi-89]{Mi-89} K. Miyake, {\em Algebraic Investigations of Hilbert's theorem 94, The principal Ideal theorem and capitulation Problem, }{Expos. Math. {\bf 7} (1989), 289-346.}

\bibitem[N.B-57]{N.B-57} N. Blackburn, {On prime-power groups in which the derived group has two generators,}
{Proc. Cambridge Phil. Soc. {\bf 53} (1957), 19-27.}

 \bibitem[Kw-80]{Kw-80}
 K.S. Williams, {\em On the evaluation of $(\varepsilon_{q_1q_2}/p)$, }{ Rocky Mountain Journal of mathematics. {Volume {\bf 10}, Number 3, Summer(1980), 559-573.}}

\bibitem[PS-71]{PS-71}
  Pierre Samuel, {\em Théorie Algébrique Des Nombres, } { Hermann Paris, deuxiéme edition  (1971). }

 \bibitem[Kub-56]{Kub-56} T. Kubota, {\em \"Uber den bizyklischen biquadratischen zahlk\"orper, }{Nagoya Math. J. 10 (1956), 65-85.}

 \bibitem[Kur-43]{Kur-43} S. Kuroda, {\em \"Uber den Dirichletschen K\"orper, }{J. Fac. Sci. Imp. Univ. Tokyo sec. 14 (1943), 383-406.}

 \bibitem[Se-70]{Se-70}
  J.Pierre Serre, {Cours d'arithmétiques, } { P.U.F  (1970). }
\bibitem[Sc-34]{Sc-34}
 A. Scholz, {\em \"Uber die L\"obarkeit der Gleichung $t^2-Du^2=-4$,}{Math. Z. {\bf 39} (1934), 95-111.}
 \bibitem[H.K-68]{H.K-68}H. Koch, {\em Zum satz von Golod-Schafarewitsch,}{ Math. Nachr.{\bf 42} (1968), 321-333.}
\bibitem[Wa-66]{Wa-66} H. Wada, {\em On the class number and the unit group of
    certain algebraic number fields}, { J. Fac. Univ. Tokyo Sect. I {\bf 13} (1966),
    201-209.}
\bibitem[Xcas-07]{Xcas-07}
Version 3, 29 June 2007, {Copyright(C) 2007 free Software Foundation, Inc. <http://fsf.org/>.}
\bibitem[Z.J-05]{Z.J-05} Z. Janko, {\em A classification of finite 2-groups with exactly three involutions}, {J. Algebra {\bf 291} (2005), 505-533.}
\end{thebibliography}

\end{document}